\documentclass[a4paper,10pt]{article}
\usepackage[utf8]{inputenc}

\usepackage{amsmath,amssymb,amsthm,amsbsy}
\usepackage[ngerman,american]{babel} 
\usepackage[subrefformat=parens,labelformat=parens]{subfig}
\usepackage{graphicx}
\usepackage{paralist} 
\usepackage{tabularx}
\usepackage{multirow} 
\usepackage{rotating}
\usepackage{booktabs} 
\usepackage{setspace} 
\usepackage{tikz}
\usepackage{pgfplots}
\usetikzlibrary{arrows,positioning,matrix,petri,calc,fit,shapes}

\usepackage{breqn} 
\usepackage{blindtext}
\usepackage{scrpage2}
\usepackage{babel}
\usepackage{changes}
\usepackage[exponent-product=\cdot, binary-units=true]{siunitx}
\usepackage{listings}	
\usepackage{lstlinebgrd}

\usepackage{authblk}

\usepackage{hyperref}

\DeclareRobustCommand{\examplesymbol}{%
  \ifmmode \tag*{$ \diamondsuit $}
  \else
    \leavevmode\unskip\penalty9999 \hbox{} \nobreak \hfill \hbox{$ \diamondsuit $}%
  \fi
}

\theoremstyle{plain}
\newtheorem{defn}{Definition}
\newtheorem{thm}[defn]{Theorem}
\newtheorem{lem}[defn]{Lemma}
\newtheorem{cor}[defn]{Corollary}

\theoremstyle{definition}

\newtheorem{rem}[defn]{Remark}

\newcommand{\calB}{\mathcal{B}}
\newcommand{\B}{\mathcal{B}}
\newcommand{\ddu}{\ddot{u}}
\newcommand{\du}{\dot{u}}
\newcommand{\dv}{\dot{v}}
\newcommand{\ddv}{\ddot{v}}
\newcommand{\bfu}{\mathbf{u}}
\newcommand{\bfdu}{\mathbf{\du}}
\newcommand{\bfddu}{\mathbf{\ddu}}
\newcommand{\bfv}{{v}}
\newcommand{\bfdv}{{\dv}}
\newcommand{\bfddv}{{\ddv}}
\newcommand{\bfw}{{w}}

\newcommand{\dphi}{{\dot{\phi}}}

\newcommand{\R}{\mathbb{R}}

\newcommand{\N}{\mathbb{N}}

\newcommand{\HB}[1]{H^{#1}_{\mathcal{B}}(\Omega)}

\DeclareMathOperator*{\esssup}{ess\,sup}

\allowdisplaybreaks

\title{Existence, Uniqueness and Regularity of Piezoelectric Partial Differential Equations}
\author{Benjamin Jurgelucks$^1$, Veronika Schulze$^1$, Tom Lahmer$^2$}
\date{%
    $^1$Paderborn University, %
    $^2$Bauhaus-University Weimar%
}

\begin{document}

\maketitle

\section*{Abstract}
Piezoelectric appliances have become hugely important in the past century and computer simulations play an essential part in the modern design process thereof.
While much work has been invested into the practical simulation of piezoelectric ceramics there still remain open questions regarding the partial differential equations governing the 
piezoceramics.

The piezoelectric behavior of many piezoceramics can be described by a second order coupled partial differential equation system. 
This consists of an equation of motion for the mechanical displacement in three dimensions and a coupled electrostatic equation for the electric potential.
Furthermore, an additional Rayleigh damping approach makes sure that a more realistic model is considered.

In this work we analyze existence, uniqueness and regularity of solutions to theses equations and give a result concerning the long-term behavior.
The well-posedness of the initial boundary value problem in a bounded domain with sufficiently smooth boundary is proved by Galerkin approximation 
in the discretized weak version, followed by an energy estimation using Gronwall inequality and using the weak limit to show the results in the infinite dimensional space.
Initial conditions are given for the mechanical displacement and the velocity.

\section{Introduction}
Piezoelectricity has become more and more important for technical purposes and innovations especially when high-frequency vibrations are to be measured or produced. 
Typical applications as actuators range from piezo-igniters over ultrasonic toothbrushes to diesel fuel injectors as well as many others, e.g., as part of intelligent sensory equipment.
The piezoelectric effect describes the transaction between electrical and mechanical energy changes of a piezoelectric sample. The effect is caused by the structure of the material and its polarization. Therefor it is clear, that the effect and its usage is material based (cf.~\cite{heywang}).\\
There are two problems which can be solved regarding the piezoelectric equations, the forward and the inverse problem. For details regarding the inverse problem and optimization of sensitivities see e.g.~\cite{feldmann},~\cite{jurgelucks}.\\
In order to design and analyse new piezoelectric devices, models are employed \cite{NanthakumarLahmer}. However for a reliable use existence, uniqueness and regularity of the solutions for these models need to be guaranteed.\\

The underlying application of the well-posedness result is a piezoelectric ceramic disc with top and bottom surface electrodes. The material parameters are extracted from real measurements for the forward simulation to compute the mechanical displacement and the electrical potential after electrical excitation.

The proof of the properties mentioned above assumes a bounded domain with sufficiently smooth boundary. Our piezoceramic and the electrodes on top and bottom fulfill these requirements.

The underlying model is linear, includes Rayleigh damping and neglects thermal effects. The behavior of the piezoelectric material can be described by a second order partial differential equation system, which defines the mechanical displacement and the electrical potential.
By an appropriate choice of the Rayleigh damping parameters, the equation of motion of the mechanical displacement is a hyperbolic partial differential equation and the electrostatic equation of the electrical potential is an elliptic partial differential equation. 
The density, the elastic stiffness, the dielectric permittivity and the piezoelectric coupling matrices are the given material components in the standard Voigt notation. 
There are several existing works on the well-posedness of the piezoelectric initial boundary problem usually without any damping models. The proof structures used in this paper are similar. Parts of our work are based on the proof presented in \cite{lahmer}. Technical details are however elaborated in more detail and some derivations are developed in a more rigorous way. 
Proofs for the static and harmonic case can be found in \cite{kaltenbacher_lahmer_mohr_pde_based} and \cite{lahmer}.

The proof is divided in four general steps. First, the system is transformed into the weak form and discretized, via Galerkin approximation. Then, via standard theory for ordinary differential equations there exist unique solutions. The finiteness of the finite dimensional solution is shown by the energy estimates via the Gronwall inequality. The weak limit of the discretized solutions provide the weak existence of a solution in infinite dimensional function spaces. The uniqueness of the solution is shown by applying the estimates to the homogeneous system and getting the trivial solution.

In the second part of the paper, Theorem \ref{theorem2} studies higher regularities for the solution of the system based on higher regularity requirements for the initial condition of the mechanical displacement, the velocity and the boundary value for the electrical potential.
Finally, a remark about the long-term behaviour of an energy functional considered in the proof of Theorem \ref{theorem1} is stated.

\section{Setting}

Before we can begin to solve any partial differential equation we must first establish an exact setup - the geometry $\Omega$, the boundary $\partial \Omega$, the boundary conditions and initial values of the partial differential equations in question. 
We consider the case of a mechanically unclamped piezoceramic which is excited by prescribing a voltage on a part of the boundary.
Let $\Omega\subseteq \R^3$ be an open domain describing the piezoelectric ceramic and let $\partial \Omega=:\Gamma$ be the nonempty boundary of $\Omega$.
The boundary is divided into nonempty, disjunct, covering subsets of $\Gamma$ (see also Fig.~\ref{fig:boundary}) which are assumed to have a positive 2D measure.
Let $\Gamma_e$ be the section of the boundary which is electrically excited, $\Gamma_g$ the section of the boundary which is grounded, $\Gamma_r=\Gamma \setminus \left(\Gamma_e \cup \Gamma_g \right)$ the remaining boundary section.
 \begin{figure}[!ht]
 \begin{center}
\includegraphics[width=0.9\textwidth]{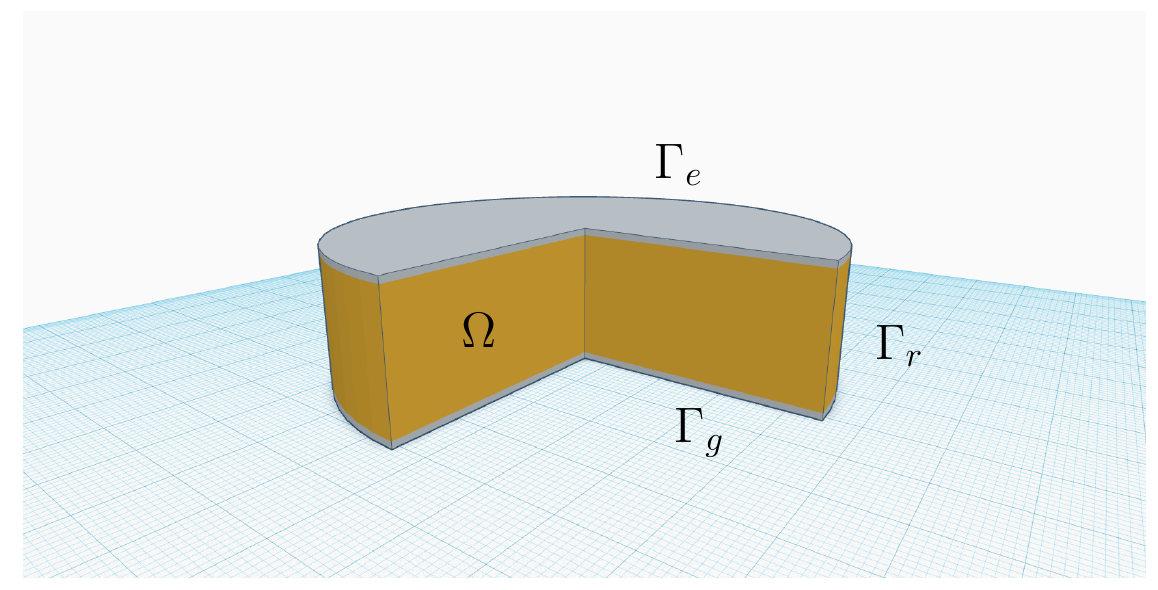}
\caption{Domain and boundaries of a piezoceramic. }
\label{fig:boundary}
\end{center}
\end{figure}

For the readers convenience the usual definitions of common function spaces which will be required later on are stated in the appendix \ref{appendix_definition}.
Only the newly defined function spaces for the considered differential equation system are described now:
\begin{equation*}
\begin{aligned}
H^1_{0,\Gamma}(\Omega)&:=  \Big\{\ \sigma_1+\sigma_2 : \sigma_1 \in H^1_0(\Omega) \text{ and } \sigma_2 \in H^1(\Omega) \Big\}\ ,\\
H^1_\B(\Omega)&:=\Big\{\ \sigma: \Omega \rightarrow \R^3 :  \|\sigma\|_{H^1_\B(\Omega)}^2:=\|\sigma\|^2_{L^2(\Omega)}+\|\B \sigma\|^2_{L^2(\Omega)}<\infty \Big\},\
\end{aligned}
\end{equation*}

where $$
	\B:= 
\begin{pmatrix}
\frac{\partial}{\partial x} & 0 & 0\\
 0 &\frac{\partial}{\partial y}  & 0\\
 0 & 0 &\frac{\partial}{\partial z}\\
0 & \frac{\partial}{\partial z} & \frac{\partial}{\partial y}\\
\frac{\partial}{\partial z} & 0 & \frac{\partial}{\partial x}\\
\frac{\partial}{\partial y} & \frac{\partial}{\partial x} & 0\\
\end{pmatrix}
$$
with $x,y,z$ refering to Cartesian coordinates. 
In this paper we denote derivatives with respect to time by the dot symbol e.g. $ \dot{ \sigma} $ and derivatives with respect to space by the nabla or $\B$ symbol, e.g. $\nabla \sigma$ or $\B \sigma$. Here $\B$ denotes the symmetric gradient in Voigt notation. It should be noted that the last three entries of the matrix vector product $\B u$ still contains the factor $2$, but for simplicity, no attention is paid here. The factor can be included in the definition of the linear strain vector $S$, where $S=\B u$.\\
All derivatives in the above are understood in the distributional sense.
In addition, the dual space of a Hilbert space $X$ is denoted by $X'$.
In particular, $H^{-1}(\Omega)$ denotes the dual space of $H^1_0(\Omega)$. 
Note that in order to simplify the notation superscripts indicating the dimension of $\bfu$ or $\B \bfu$, which are 3 and 6 respectively, are omitted.
This is reasonable as the vectorial scalar product inside $\int_{\Omega} \sigma^T\sigma \, d\Omega$ always returns a scalar no matter what dimensions $\sigma$ has.

Let $\vec{n}:=(n_x, \, n_y,\, n_z)$ be the normal vector and 
$$
	\mathcal{N}:= 
\begin{pmatrix}
n_x & 0 & 0\\
 0 &n_y  & 0\\
 0 & 0 &n_z\\
0 & n_z & n_y\\
n_z & 0 & n_x\\
n_y & n_x & 0\\	
\end{pmatrix}.
$$

\begin{defn}
The material parameters $c^E, \epsilon^S$ and $e$ $([c^E]=N\cdot m^{-2},\ [\epsilon^S]=F\cdot m^{-1},\ [e]=C\cdot m^{-2})$ are given by
\begin{eqnarray*}
c^E:=
\left( \begin{array}{cccccc}
c_{11} & c_{12} & c_{13}&0&0&0\\
c_{12} & c_{11} & c_{13}&0&0&0 \\
c_{13} & c_{13} & c_{33} &0&0&0 \\
0&0&0&c_{44}&0&0\\
0&0&0&0&c_{44}&0\\
0&0&0&0&0&\frac{1}{2}(c_{11}-c_{12})\\
\end{array} \right)\in \R^{6 \times 6}
\end{eqnarray*}
\begin{eqnarray*}
\epsilon^S:=
\left( \begin{array}{ccc}
\epsilon_{11} &0 & 0\\
0 & \epsilon_{11} & 0 \\
0 & 0 & \epsilon_{33} 
\end{array} \right)\in \R^{3 \times 3}
\end{eqnarray*}
\begin{eqnarray*}
e:=
\left( \begin{array}{cccccc}
0 &0 & 0& 0& e_{15}& 0\\
0 &0 & 0& e_{15}& 0& 0\\
e_{13} &e_{13} & e_{33}& 0& 0& 0\\
\end{array} \right)\in \R^{6 \times 3}.
\end{eqnarray*}
The material parameters are said to fulfill nonnegativity conditions if $c^E$ and $\epsilon^S$ are positive definite matrices.
\label{nonnegativitiyconditions}
\end{defn}

The three dimensional transient linear piezoelectric equations with Rayleigh damping parameters $\alpha,\beta>0$ (chosen sufficiently large enough so that the system is parabolic) and density $\rho >0$ describing the mechanical displacement $\bfu \in \R^3$ and the electrical potential $\phi \in \R$
with given boundary conditions are stated as: 
\begin{eqnarray*}
     \rho \bfddu(t) + \alpha \rho \bfdu(t)-\calB^T \left( c^E \calB \bfu(t) + \beta c^E \calB \bfdu(t) + e^T \nabla\phi(t)\right)  & = & 0 \text{  in } \Omega \times [0,T],\\
    -\nabla \cdot \left(e \calB \bfu(t) -\epsilon^S \nabla \phi(t) \right) & = & 0 \text{  in } \Omega\times [0,T], \\
		    \phi(t) & = & 0 \text{  on } \Gamma_g\times [0,T], \\
				\phi(t) & = & \phi^e(t) \text{  on }\Gamma_e\times [0,T], \\
				\vec{n} \cdot  \left(e \calB \bfu(t) -\epsilon^S \nabla \phi(t) \right) & = & 0 \text{  on } \Gamma_r \times [0,T],\\
				\mathcal{N}^T   \left( c^E \calB \bfu(t) + \beta c^E \calB \bfdu(t) + e^T \nabla\phi(t)\right) & = & 0 \text{  on } \partial\Omega\times [0,T], \\
				\bfu(0)& = & \bfu_0, \\
				\bfdu(0)& = & \bfu_1.
\end{eqnarray*}

The weak form of the equations above can easily be obtained \cite{lahmer} by testing with appropriate functions $v\in\R^3$ (for the first line) and $w\in\R$ (for the second line), integration by parts and using
boundary conditions:
$$
\int_\Omega \left(\B^T \sigma \right)^Tv \,d\Omega=-\int_{\Omega} \sigma^T \B v \,d\Omega + \int_{\partial \Omega} \left(\mathcal{N}^T \sigma \right)^T v \,d\Omega.
$$

First, we use a \emph{Dirchlet lift ansatz} to homogenize the Dirichlet boundary condition for $\phi(t)$:
Let $t\in [0,T]$ and let $\chi \in H^1(\Omega)$ where $\chi\vert_{\Gamma_g}=0$ and $\chi\vert_{\Gamma_e}=1 $. 
Such a $\chi$ exists if we assume that $\Omega$ is at least a Lipschitz domain. Let $\phi(t)$ consist of two parts $\phi(t)=\phi_0(t) + \phi_{\phi^e}(t)$ where $\phi_0(t)\in H^1_0(\Omega)$ and $ \phi_{\phi^e}(t)\in H^1(\Omega)$.
We then rewrite $\phi_{\phi^e}(t)=\phi^e(t)\chi$. Therefor we set $\phi_0(t):=\phi(t)-\phi^e(t)\chi$.

As $\phi^e(t)$ is a given value $\phi^e(t)\chi$ can be taken out of the left hand side of the weak form and added to the right hand side. 
The weak form of the piezoelectric system  for all $ t \in [0,T] \text{ a.e.}$ and  for all test functions $(v,w) \in \HB{1} \times H^1_0(\Omega)$ is given by

\begin{equation}
\begin{aligned}
&\int_\Omega \rho \bfddu^T v \,d\Omega +\alpha \int_{\Omega} \rho \bfdu^Tv \,d\Omega +\int_\Omega \left(c^E \calB \bfu\right)^T\calB v \,d\Omega   + \beta\int_\Omega \left(c^E \calB \bfdu\right)^T\calB v \,d\Omega  \\
+& \int_\Omega \left( e^T\nabla\phi_0\right)^T \calB v \,d\Omega + \int_\Omega \left( e \calB \bfu\right)^T\nabla w \,d\Omega -  \int_\Omega \left(\epsilon^S \nabla \phi_0\right)^T\nabla w \,d\Omega \\
= &\, \phi^e \int_{\Omega} -(e^T \nabla \chi)^T\calB v  +  (\epsilon^S \nabla \chi)^T \nabla w \,d\Omega.
\label{eq_schwache_form_erstes_auftreten}
\end{aligned}
\end{equation}

Note that in light of \cite[Thm. 2 in section 5.9.2]{evans} it makes sense to demand $\bfu(0)=\bfu_0$ and $\bfdu(0)=\bfu_1$. See also the only remark in \cite[section 7.2.1]{evans}.

\section{Existence, uniqueness and regularity of solutions}
Before we attempt to show existence, uniqueness and regularity of solutions some additional tools are required:

\begin{lem}(Young inequality)\\
Let $1<p, q< \infty, \frac{1}{p}+\frac{1}{q}=1$. 
Then for $a,b>0$ the following inequality holds:
$$
ab \leq \frac{a^p}{p}+ \frac{b^q}{q}.
$$
\end{lem}
\begin{proof}
See \cite[Appendix B.2]{evans}.
\end{proof}

\begin{lem}(H\"older inequality) \\
Let $1\leq p, q\leq \infty, \frac{1}{p}+\frac{1}{q}=1$. 
Then for $u\in L^p(\Omega),v\in L^q(\Omega) $ the following inequality holds:
$$
 \int_\Omega |uv| \, dx \leq \|u\|_{L^p(\Omega)} \|v\|_{L^q(\Omega)}.
$$   
\end{lem}
\begin{proof}
See \cite[Appendix B.2]{evans}.
\end{proof}

\begin{rem}
The latter two inequalities are especially true for $p=q=2$.
The latter inequality is then known as Cauchy--Schwarz (C.S.) inequality.
\end{rem} 

\begin{lem}(Gronwall inequality, integral form) \\
 \begin{enumerate}
 \item[a)] 
Let $\eta:[0,T]\rightarrow \R^{\ge 0}$ be a nonnegative, summable function on $[0,T]$, which satisfies for
almost every $t$ the differential inequality
$$
\eta(t) \leq C_1 \int_0^t \eta(s) \, ds +C_2
$$
for constants $C_1,C_2 \geq 0$. Then
$$
\eta(t) \leq C_2\left( 1+C_1t e^{C_1 t}    \right)
$$
for a.e. $0 \leq t \leq T$.
\item[b)]
In particular, if 
$$
\eta(t) \leq C_1 \int_0^t \eta(s) ds
$$
for a.e. $t\in [0,T]$, then
$$
\eta(t)=0 \text{ a.e. }
$$
\end{enumerate}
\end{lem}
\begin{proof}
See \cite[Appendix B.2]{evans}.
\end{proof}

\begin{rem}(Sufficiently smooth boundary) \\
We say the boundary $\partial \Omega$ is sufficiently smooth if it permits application of the trace theorem (cf. \cite{evans}).

Thus, a $C^1-$boundary is sufficient. However, it is possible to utilize a variation of the trace theorem under less strict requirements (cf. \cite{ding_trace}).
We note that the boundary for our specific application (see Fig.~\ref{fig:boundary}) satisfies the special Lipschitz condition stated in Definition 5 of \cite{ding_trace}
and thus it appears that it can also be considered sufficiently smooth.
\label{def_smooth_bound}
\end{rem}

A proof for the following theorem was first given in \cite{lahmer}. 
The proof given there is also heavily oriented on work of \cite{japaner} which itself is based on \cite{Melnik}.
Here we present a proof with the similar essential steps as in the other given proofs, but with more technical details and necessary exact descriptions.\\

The proof follows the usual guideline as seen for many partial differential equations (e.g. \cite[p.~353]{evans}):
We get existence and uniqueness of a weak solution by the usual procedure:
\begin{enumerate}
 \item Discretization via Galerkin approximation of infinite dimensional function spaces,
 \item energy estimates via Gronwall inequality in discretized space which provide finiteness of the discretized solution,
 \item weak limit of discretized solution provides weak existence of a solution in infinite dimensional function space,
 \item uniqueness of the solution is shown by applying estimates to the difference $w:=w_1-w_2$ of two solutions $w_1$ and $w_2$. 
Thus, the only solution to the homogeneous case is the trivial solution.
 
\end{enumerate}
\begin{thm}  \, \quad \\
Let $\Omega \subseteq \R^3$ be a bounded domain with sufficiently smooth boundary as specified in Remark \ref{def_smooth_bound}. Let the real valued material parameters $c^E, e$ and $\epsilon^S$ be defined as in Def.~\ref{nonnegativitiyconditions} and let $c^E$ and $\epsilon^S$ be symmetric and positive definite. 
The Rayleigh coefficients $\alpha$ and $\beta$ are assumed to be nonnegative. 
Let $T>0$ and $\rho >0$. \\
Then there exists a $C>0$ such that for any $\bfu_0\in \HB{1}, \bfu_1 \in L^2(\Omega)$ and $\phi^e \in H^1(0,T;H^{1/2}(\Gamma_e))$ there exists a unique solution
\begin{equation}
(\bfu,\phi)\in L^{\infty}(0,T;\HB{1}) \times L^{\infty}(0,T;H^1_{0,\Gamma}(\Omega))
\label{raeume1}
\end{equation}
with 
\begin{equation}
\bfdu \in L^{\infty}(0,T;L^2({\Omega})) \text{ and } \bfddu \in L^2(0,T;(H^{1}_\calB(\Omega))')
\label{raeume2}
\end{equation}
to Eq.~\eqref{eq_schwache_form_erstes_auftreten} 
satisfying the initial conditions
$$
\bfu(0)=\bfu_0,\quad \bfdu(0)=\bfu_1 \text{  on } \Omega
$$
and the following estimate holds:

\begin{equation}
\begin{aligned}
&\|\bfu\|_{L^{\infty}(0,T;\HB{1})} 
+ \|\bfdu\|_{L^{\infty}(0,T;L^2(\Omega))}+\|\bfddu\|_{L^{2}(0,T;(H^{1}_\calB(\Omega))')}
+\|\phi\|_{L^{\infty}(0,T;H^{1}_{0,\Gamma}(\Omega))} \\
\leq&\  C \left( \|\bfu_0\|_{\HB{1}} +\|\bfu_1\|_{L^2(\Omega)} + \|\phi^e\|_{H^1(0,T;H^{1/2}(\Gamma_e))}     \right).
\label{chp_2_weak_form}
\end{aligned}
\end{equation}
\label{theorem1}
\end{thm}

\begin{proof}  
Note that many concepts of this proof are taken from \cite[chapter 7]{evans} and information regarding involved spaces can be found in \cite{adams}.

In the following constants denoted by the letter $C$ or $\tilde{C}$ are used. 
Unless explicitly specified otherwise we note that all these constants are positive $C_i>0\, ,i\ge 1$.

Weak solutions are functions $\bfu,\bfdu, \bfddu$ and $\phi_0$ as in Eq.~\eqref{raeume1} and Eq.~\eqref{raeume2} where $\phi_0=\phi+\phi_{\phi_e}$ such that 
for almost all $t \in [0,T]$
for all $(v,w)\in \HB{1} \times H^1_{0}(\Omega)$  the following equation holds:

\begin{equation}
\begin{aligned}
&\left< \rho \bfddu(t), v \right> +\alpha \left< \rho \bfdu(t),v \right> +\left< c^E \calB \bfu(t),\calB v \right>   + \beta\left< c^E \calB \bfdu(t),\calB v \right>  \\
+& \left<  e^T\nabla\phi_0(t), \calB v \right> + \left<  e \calB \bfu(t),\nabla w \right> -  \left< \epsilon^S \nabla \phi_0(t),\nabla w \right> \\
= & \left<f(t),v \right> +\left<g(t),w \right>
\end{aligned}
\label{eq_schwache_form}
\end{equation}
with 
$$\left<f(t),v\right>:=-\phi^e(t) \int_{\Omega} (e^T \nabla \chi)^T\calB v \,d\Omega ,$$
and 
$$\left<g(t),w\right>:=\phi^e(t) \int_{\Omega} (\epsilon^S \nabla \chi)^T \nabla w \,d\Omega .$$
Note that by the Riesz representation theorem there exists a unique representation for the latter functionals as an inner product, i.e. $\left<f,\cdot\right>$ and $\left<g,\cdot\right>$. 
As is common in the field of partial differential equation for convenience we will also use the same symbols $f$ and $g$ to refer to the Riesz-representative as well as the functionals $\left<f,\cdot\right>$ and $\left<g,\cdot\right>$.
Furthermore, we remember that $\chi\in H^1(\Omega)$ and that $\epsilon^S$, $e$ are constant. The integrals of the right hand side $\int_{\Omega} (e^T \nabla \chi)^T\calB v \,d\Omega$, $\int_{\Omega} (\epsilon^S \nabla \chi)^T \nabla w \,d\Omega$ are finite, their values $c_1(\Omega), c_2(\Omega) < \infty$ depend, e.g., ~only on domain $\Omega$ but not on time $t$.
Thus, by integrating this constant value over time we can estimate the Bochner-space norm of $f$ by 
$$
\|f\|_{H^1(0,T;(H^{1}_\calB(\Omega))'
)} \leq c_1(\Omega) \|\phi^e\|_{H^1(0,T)},
$$
and analogously we get
$$
\|g\|_{H^1(0,T;H^{-1}(\Omega))} \leq c_2(\Omega) \|\phi^e\|_{H^1(0,T)}.
$$

\textbf{Phase 1: Galerkin approximation}\\
The weak form is tested with test functions 
$\bfv_j\in H^1_\calB(\Omega)$ and $\bfw_j\in H^1_{0}(\Omega), \quad j \in \N$, with
$$
\bfu(t) \approx \bfu_m(t) = \sum_{j=1}^m u_m^j(t) \bfv_j,
$$
and
$$
\phi_0(t) \approx \phi_m(t) = \sum_{j=1}^m \phi_m^j(t) \bfw_j,
$$
where '$\approx$' is to be understood in the sense of an orthogonal projection in the appropriate spaces.
The finite dimensional spaces spanned by the test functions are defined as 
$$V_m:=\operatorname{span}\{\bfv_1,\dots,\bfv_m\} \quad \text{ and } \quad  W_m:=\operatorname{span}\{\bfw_1,\dots,\bfw_m\}.$$ We can assume that the dimension of the test function spaces $\operatorname{dim}(V_m)=\operatorname{dim}(W_m)=m$ are the same, for  $V_m$ in each vectorial component. So the test functions can be selected to be linearly independent.
Furthermore, the functions can be chosen such that
$$\overline{\bigcup_{m=1}^\infty V_m}=H^1_\calB(\Omega)\quad \text{and}\quad 
\overline{\bigcup_{m=1}^\infty W_m}=H^1_{0}(\Omega).$$

Then via standard theory for ordinary differential equations (see e.g. \cite{evans} or \cite{leugering}) for all $m\in \N$ and for all $(\bfv_m,w_m)\in V_m \times W_m$
there exists a unique solution
$$
(\bfu_m,\phi_m) \in C^2([0,T];V_m) \times C([0,T];W_m)$$
to the discretized version of Eq.~\eqref{eq_schwache_form} that fulfills the initial conditions 
$\bfu_m(0)=(\bfu_{0})_m,$ $\bfdu_m(0)=(\bfu_{1})_m$. 
For more information on Sobolev spaces involving time and space see also \cite[section 5.9.2]{evans}.

\textbf{Phase 2: Energy estimates}\\
The aim of this phase is to use Gronwall inequality to show an energy estimate from which the finiteness of
the finite dimensional solutions $(\bfu_m(t),\phi_m(t))$ in $L^\infty(0,T;H^1_\B(\Omega)) \times L^\infty(0,T;H^1_{0}(\Omega))$, $\bfdu_m$ in $L^\infty(0,T;L^2(\Omega))$ and $\bfddu_m$ in\\ \hfill $L^2(0,T;(H^{1}_\calB(\Omega))')$ can be deduced:
\vspace{3mm}

Let 
$$
\eta(t):=\left(\|\bfdu_m(t)\|^2_{L^2(\Omega)} +\|\bfu_m(t)\|^2_{H^1_\B(\Omega)} +\|\phi_m(t)\|^2_{H^1_{0}(\Omega)} \right).
$$
In order to use the Gronwall inequality we must show that there are constants $p,q\ge0$ such that $\eta(t) \leq p \int_0^t \eta(s) \, ds +q$ holds.
If this condition is true, then it can be shown that

\begin{equation}
\begin{aligned}
&\|\bfdu_m(t)\|^2_{L^2(\Omega)} +\|\bfu_m(t)\|^2_{H^1_\B(\Omega)} +\|\phi_m(t)\|^2_{H^1_{0}(\Omega)} \\
\leq & \left(  1+ p t e^{pt}  \right)     \Big( \|\bfdu_m(0)\|^2_{L^2(\Omega)} +\|\bfu_m(0)\|^2_{H^1_\B(\Omega)} +\|\phi_m(0)\|^2_{H^1_{0}(\Omega)} \\
&+ \|f\|^2_{L^2(0,T;(H^{1}_\calB(\Omega))')} +\|g\|^2_{H^1(0,T;H^{-1}(\Omega))}\Big)
\end{aligned}
\label{eq_166}
\end{equation}
holds almost everywhere in $[0,T]$. 
Thus, this must also be true for the essential supremum over $0\leq t \leq T$ and we will get finiteness in the $L^\infty(0,T;X)$ norm for the appropriate sub-spaces $X$.
In order to show the requirement we consider the following:

First, the discretized version of the weak form Eq.~\eqref{eq_schwache_form} is supposed to hold for all test functions $(\bfv_m,w_m)$.
Thus, it should also hold for $(\bfdu_m(t),0)$:

\begin{equation*}
\begin{aligned}
 &   \left<\rho \bfddu_m(t), \bfdu_m(t)\right>
+ \alpha \left<\rho \bfdu_m(t), \bfdu_m(t)\right> +\left<c^E \B \bfu_m(t), \B \bfdu_m(t)\right> \\
+& \beta \left<c^E \B \bfdu_m(t), \B \bfdu_m(t)\right>
+ \left<e^T \nabla \phi_m(t), \B \bfdu_m(t)\right> 
= \left< f(t), \bfdu_m(t)\right>.
\end{aligned}
\end{equation*}
By transposing the inner product and direct computation it is easy to see that one can swap the placement of constant scalars or matrices such as $\rho,\epsilon^S,c^E$ etc. (which are symmetric) in this bilinear form, e.g. the following holds:
$$\left<c^E \B \bfu_m(t), \B \bfdu_m(t)\right>=\left<(c^E)^T \B \bfdu_m(t), \B \bfu_m(t)\right>=\left<c^E \B \bfdu_m(t), \B \bfu_m(t)\right>.$$

Thus, by bilinearity of the inner product
$$
2\left<c^E \B \bfdu_m(t), \B \bfu_m(t)\right> = \frac{d}{dt} \left<c^E \B \bfu_m(t), \B \bfu_m(t)\right>.
$$

Hence, the above equation simplifies to 

\begin{equation}
\begin{aligned}
 & \frac{1}{2} \frac{d}{dt}\left(  \left<\rho \bfdu_m(t), \bfdu_m(t)\right> +\left<c^E \B \bfu_m(t), \B \bfu_m(t)\right>    \right)
+ \alpha \left<\rho \bfdu_m(t), \bfdu_m(t)\right>\\
 +& \beta \left<c^E \B \bfdu_m(t), \B \bfdu_m(t)\right>
+ \left<e^T \nabla \phi_m(t), \B \bfdu_m(t)\right>
= \left< f(t), \bfdu_m(t)\right>.
\end{aligned}
\label{eq_157}
\end{equation}
Now we differentiate the weak form Eq.~\eqref{eq_schwache_form} with respect to $t$ and test it with $(0,\phi_m(t))$, taking into account that the test functions $\bfv_m,w_m$ do not depend on time $t$, therefor the time derivatives  $\dot{\bfv}_m,\dot{w}_m \equiv 0$: 

\begin{equation}
\begin{aligned}
  \left<e \B \bfdu_m(t), \nabla \phi_m(t) \right> -  \frac{1}{2}\frac{d}{dt}\left<\epsilon^S \nabla \phi_m(t) ,\nabla \phi_m(t)\right>
	= \left< \dot g(t),\phi_m(t)\right>. \\
\end{aligned}
 \label{eq_158}
\end{equation}

A subtraction of Eq.~\eqref{eq_157} and Eq.~\eqref{eq_158} gives
\begin{equation}
\begin{aligned}
 & \frac{1}{2} \frac{d}{dt}\left(  \left<\rho \bfdu_m(t), \bfdu_m(t)\right> +\left<c^E \B \bfu_m(t), \B \bfu_m(t)\right>  
+\left<\epsilon^S \nabla \phi_m(t) ,\nabla \phi_m(t)\right>   \right) \\
+& \alpha \left<\rho \bfdu_m(t), \bfdu_m(t)\right> + \beta \left<c^E \B \bfdu_m(t), \B \bfdu_m(t)\right> \\
=& \left< f(t), \bfdu_m(t)\right> -\left< \dot g(t),\phi_m(t)\right>.\\
\end{aligned}
\label{eq_159}
\end{equation}
The last equation Eq.~\eqref{eq_159} is integrated with respect to $t$.
\begin{equation}
\begin{aligned}
 \mathcal{F}_l(t):=&  \left<\rho \bfdu_m(t), \bfdu_m(t)\right> +\left<c^E \B \bfu_m(t), \B \bfu_m(t)\right>  
+\left<\epsilon^S \nabla \phi_m(t) ,\nabla \phi_m(t)\right>  \\
&+ 2 \alpha \int_0^t\left<\rho \bfdu_m(s), \bfdu_m(s)\right> ds+2 \beta \int_0^t\left<c^E \B \bfdu_m(s), \B \bfdu_m(s)\right> ds \\
=&  \left<\rho \bfdu_m(0), \bfdu_m(0)\right> +\left<c^E \B \bfu_m(0), \B \bfu_m(0)\right>  +\left<\epsilon^S \nabla \phi_m(0) ,\nabla \phi_m(0)\right> \\
&+ 2\int_0^t\left< f(t), \bfdu_m(s)\right>ds -2\int_0^t\left< \dot g(t),\phi_m(s)\right>ds =:\mathcal{F}_r(t).
\end{aligned}
\label{eq_160}
\end{equation}
Hence, in short we can write $$\mathcal{F}_l(t)=\mathcal{F}_r(t).$$
Now the aim is to use this equation to show that the requirements for the Gronwall inequality are met.

We start by showing that the left-hand side $\mathcal{F}_l(t)$ of Eq.~\eqref{eq_160} has a lower bound.
With $\lambda_{1,mech}$ the smallest eigenvalue of $c^E$ (which is strictly positive) one estimates
\begin{equation}
\begin{aligned}
\int_\Omega \left(\B\bfu_m(t)\right)^Tc^E\B\bfu_m(t) \,d\Omega & \geq \lambda_{1,mech} \int_\Omega \left(\B\bfu_m(t)\right)^T\B\bfu_m(t) \,d\Omega \\
                                                           &= \lambda_{1,mech} \|\B\bfu_m(t)\|^2_{L^2(\Omega)} \\
	 &= \lambda_{1,mech} \left(     \|\bfu_m(t)\|^2_{H^1_\B(\Omega)}    -\|\bfu_m(t)\|^2_{L^2(\Omega)}       \right) .
\end{aligned}
\label{eq_161}
\end{equation}
With $\lambda_{1,elec}$ the smallest eigenvalue of $\epsilon^S$ (which is strictly positive) one estimates
\begin{equation*}
\begin{aligned}
\int_\Omega \left(  \nabla \phi_m(t) \right)^T \epsilon^S  \nabla \phi_m(t) \,d\Omega &  \geq \lambda_{1,elec} \int_\Omega \left(  \nabla \phi_m(t) \right)^T \nabla \phi_m(t) \,d\Omega \\
         &= \lambda_{1,elec} \| \nabla \phi_m(t)\|^2_{L^2(\Omega)}. \\
\end{aligned}
\end{equation*}
From the Poincar\'e inequality (see e.g. \cite{schweizer}), we obtain $c_1,c_2 \in \R$ such that
\begin{equation}
\begin{aligned}
&\int_\Omega \left(  \nabla \phi_m(t) \right)^T \epsilon^S  \nabla \phi_m(t) \,d\Omega
\geq \lambda_{1,elec} \| \nabla \phi_m(t)\|^2_{L^2(\Omega)} \\
=& (1+c_2)c_1 \| \nabla \phi_m(t)\|^2_{L^2(\Omega)} = c_1\left(  \underbrace{c_2\| \nabla \phi_m(t)\|^2_{L^2(\Omega)}}_{\geq \| \phi_m(t)\|^2_{L^2(\Omega)}} +\| \nabla \phi_m(t)\|^2_{L^2(\Omega)}\right) \\
\geq&\ C_{elec}      \|\phi_m(t)\|^2_{H^1_{0}(\Omega)}.  \\
\end{aligned}
\label{eq_162}
\end{equation}

By nonnegativity of $\rho,\alpha,\beta$ and the two inequalities Eq.~\eqref{eq_161} and Eq.~\eqref{eq_162} one can now estimate 
\begin{equation*}
\begin{aligned}
 C_1\left(\|\bfdu_m(t)\|^2_{L^2(\Omega)} +\|\bfu_m(t)\|^2_{H^1_\B(\Omega)} 
+\|\phi_m(t)\|^2_{H^1_{0}(\Omega)}-c_{mech}\|\bfu_m(t)\|^2_{L^2(\Omega)} \right) \leq \mathcal{F}_l(t) 
\end{aligned}
\label{eq_212}
\end{equation*}
with a positive constant $C_1>0$.
Furthermore, by the inequalities Eq.~\eqref{eq_161} and Eq.~\eqref{eq_162} and Cauchy--Schwarz and Young inequalities the right hand side $\mathcal{F}_r(t)$ can be bounded from above with $c,\tilde{c}>0$:

\begin{equation*}
\begin{aligned}
\mathcal{F}_r(t)=&   \underbrace{\left<\rho \bfdu_m(0), \bfdu_m(0)\right>}_{= \rho \|\bfdu_m(0)\|^2_{L^2(\Omega)}} +\underbrace{\left<c^E \B \bfu_m(0), \B \bfu_m(0)\right>}_{\leq c \|\bfu_m(0)\|^2_{H^1_\B(\Omega)}}  +\underbrace{\left<\epsilon^S \nabla \phi_m(0) ,\nabla \phi_m(0)\right>}_{\leq \tilde{c}\|\phi_m(0)\|^2_{H^1_{0}(\Omega)} }  \\
&+ 2\int_0^t\left< f(s), \bfdu_m(s)\right>\,ds -2\int_0^t\left< \dot g(s),\phi_m(s)\right>\,ds\\
\\
\leq &\ \widehat{C_2}\left( \|\  \bfdu_m(0)\|^2_{L^2(\Omega)} +  \|\bfu_m(0)\|^2_{H^1_\B(\Omega)}+ \|\phi_m(0)\|^2_{H^1_{0}(\Omega)}     \right)\\
&+ 2\int_0^t\left|\left< f(s), \bfdu_m(s)\right>\right|\,ds 
+2\int_0^t\left|\left< \dot g(s),\phi_m(s)\right>\right|\, ds\\
\\
\leq &\ \widehat{C_2}\left( \|\  \bfdu_m(0)\|^2_{L^2(\Omega)} +  \|\bfu_m(0)\|^2_{H^1_\B(\Omega)}+ \|\phi_m(0)\|^2_{H^1_{0}(\Omega)}     \right)\\
&+ \int_0^t \underbrace{ \|\bfu_m(s)\|^2_{H^1_{\calB}(\Omega)}+ \|\bfdu_m(s)\|^2_{L^2(\Omega)} +\|\phi_m(s)\|^2_{H^1_0(\Omega)} }_{\ge 0} \, ds \\
&+ 2\|f\|^2_{L^2(0,T;(H^{1}_\calB(\Omega))')}  + 2\|g\|^2_{H^1(0,T;H^{-1}(\Omega))}    \\
\end{aligned}
\end{equation*}
Hence, we get
\begin{equation}
\begin{aligned}
\mathcal{F}_r(t)
\leq&\  C_2\left(  \|\bfdu_m(0)\|^2_{L^2(\Omega)} + \|\bfu_m(0)\|^2_{H^1_\B(\Omega)}  +\|\phi_m(0)\|^2_{H^1_{0}(\Omega)}   \right)\\
&+\int_0^t \left(  \|\bfdu_m(s)\|^2_{L^2(\Omega)} + \|\bfu_m(s)\|^2_{H^1_\B(\Omega)}  +\|\phi_m(s)\|^2_{H^1_{0}(\Omega)}   \right)ds\\
&+ \|f\|^2_{L^2(0,T;(H^{1}_\calB(\Omega))')}  + \|g\|^2_{H^1(0,T;H^{-1}(\Omega))}    \\
\end{aligned}
\end{equation}
with a positive constant $C_2>0$.
As $\mathcal{F}_l(t)=\mathcal{F}_r(t)$ it is now clear that
\begin{equation}
\begin{aligned}
&\ C_1\left(\|\bfdu_m(t)\|^2_{L^2(\Omega)} +\|\bfu_m(t)\|^2_{H^1_\B(\Omega)} +\|\phi_m(t)\|^2_{H^1_{0}(\Omega)}-c_{mech}\|\bfu_m(t)\|^2_{L^2(\Omega)} \right)\\
\leq&\  C_2\left(  \|\bfdu_m(0)\|^2_{L^2(\Omega)} + \|\bfu_m(0)\|^2_{H^1_\B(\Omega)}  +\|\phi_m(0)\|^2_{H^1_{0}(\Omega)}   \right)\\
&+\int_0^t \left(  \|\bfdu_m(s)\|^2_{L^2(\Omega)} + \|\bfu_m(s)\|^2_{H^1_\B(\Omega)}  +\|\phi_m(s)\|^2_{H^1_{0}(\Omega)}   \right)ds\\
&+ \|f\|^2_{L^2(0,T;(H^{1}_\calB(\Omega))')}  + \|g\|^2_{H^1(0,T;H^{-1}(\Omega))}.    \\
\end{aligned}
\label{eq_214_old}
\end{equation}
Utilizing the inequality $
\| u(t)\|_{L^2(\Omega)}^2 \leq 2 \|u(0)\|_{L^2(\Omega)}^2 + 2T\int_0^t \| \dot{u}(s) \|^2_{L^2(\Omega)} \, ds
$ (see \cite{wloka} p.~425) for $T$ large enough, we can remove $c_{mech}\| \bfu_m(t) \|^2_{L^2(\Omega)}$ from the left hand side of the inequality to obtain:
\begin{equation}
\begin{aligned}
&\ C_1\left(\|\bfdu_m(t)\|^2_{L^2(\Omega)} +\|\bfu_m(t)\|^2_{H^1_\B(\Omega)} +\|\phi_m(t)\|^2_{H^1_{0}(\Omega)} \right)\\ 
\leq&\ C_3\left(  \|\bfdu_m(0)\|^2_{L^2(\Omega)} + \|\bfu_m(0)\|^2_{H^1_\B(\Omega)}  +\|\phi_m(0)\|^2_{H^1_{0}(\Omega)}   \right)\\
&+C_4\int_0^t \left(  \|\bfdu_m(s)\|^2_{L^2(\Omega)} + \|\bfu_m(s)\|^2_{H^1_\B(\Omega)}  +\|\phi_m(s)\|^2_{H^1_{0}(\Omega)}   \right)ds\\
&+ \|f\|^2_{L^2(0,T;(H^{1}_\calB(\Omega))')}  + \|g\|^2_{H^1(0,T;H^{-1}(\Omega))}    \\
\end{aligned}
\label{eq_214}
\end{equation}
where $C_4>0$ now also depends on the fixed value $T$.

Let
$$
\eta(t):=\|\bfdu_m(t)\|^2_{L^2(\Omega)} +\|\bfu_m(t)\|^2_{H^1_\B(\Omega)} +\|\phi_m(t)\|^2_{H^1_{0}(\Omega)} 
$$
and let $$\tilde{C_2}:= \frac{1}{C_1}\left(C_3 \eta(0) + \|f\|^2_{L^2(0,T;(H^{1}_\calB(\Omega))')} +\|g\|^2_{H^1(0,T;H^{-1}(\Omega))}\right)  \ge 0. $$
Then the above inequality simplifies to
$$
 \eta(t) \leq \frac{C_4}{C_1} \int_0^t \eta(s) \, ds + \tilde{C_2}.
$$
Hence, all requirements for Gronwall inequality have been shown to hold and it can now be safely applied and the result simplified to:
\begin{equation}
\begin{aligned}
&\|\bfdu_m(t)\|^2_{L^2(\Omega)} +\|\bfu_m(t)\|^2_{H^1_\B(\Omega)} +\|\phi_m(t)\|^2_{H^1_{0}(\Omega)} \\
\leq & \left(  \frac{\tilde{C_3}}{C_1} +\frac{C_4 \tilde{C_3}}{C_1^2}te^{ \frac{C_4}{C_1}t }  \right)
     \Big( \|\bfdu_m(0)\|^2_{L^2(\Omega)} +\|\bfu_m(0)\|^2_{H^1_\B(\Omega)} +\|\phi_m(0)\|^2_{H^1_{0}(\Omega)} \\
&+ \|f\|^2_{L^2(0,T;(H^{1}_\calB(\Omega))')} +\|g\|^2_{H^1(0,T;H^{-1}(\Omega))}\Big)
\end{aligned}
\label{eq_166}
\end{equation}
holds almost everywhere in $[0,T]$.

We will return to this inequality shortly after considering the bilinear form
\begin{equation}
A:H^1_{0}(\Omega) \times H^1_{0}(\Omega) \rightarrow \R, \quad A(\phi_m(t),w) :=\left< \epsilon^S \nabla \phi_m(t), \nabla w\right> 
\label{bilinear_form}
\end{equation}
and the continuous linear functional on $H^1_{0}(\Omega)$ for a fixed $\bfu_m(t)$
$$
b(w):=\left< e\B \bfu_m(t),\nabla w \right> - \left< g(t), w \right>
$$
which together form the weak form Eq.~\eqref{eq_schwache_form} tested by $(0,w)$.
This bilinear form $A$ is coercive (inequality Eq.~\eqref{eq_162}) and continuous:
\begin{equation*}
\begin{aligned}
\left|\left< \epsilon^S \nabla \phi_m(t), \nabla w\right> \right|
\leq &\lambda_{max} \left|\left< \nabla \phi_m(t), \nabla w\right>\right| \\
\leq &\lambda_{max} \|\nabla \phi_m(t)\|_{L^2(\Omega)} \cdot \| \nabla w \|_{L^2(\Omega)} \\
\leq &\lambda_{max} \| \phi_m(t)\|_{H^1_{0}(\Omega)} \cdot \|  w \|_{H^1_{0}(\Omega)}\\
\end{aligned}
\end{equation*}

Using the Lax--Milgram lemma and the Young inequality we get the estimate for $A(\phi_m(t),w) =b(w)$ $\forall w \in H^1_0(\Omega)$:
\begin{equation}
\begin{aligned}\label{lax_milgram}
   \|\phi_m(t)\|_{H^1_{0}(\Omega)}^2 \leq &\ \tilde{M}  \|b\|_{H^{-1}(\Omega)}^2  \\
= & \ \tilde{M}  \sup_{\| w\|_{H^{1}_{0}(\Omega)} \leq 1  } \|b(w)\|_{H^{1}_{0}(\Omega)}^2 \\
= &\ \tilde{M}  \sup_{\| w\|_{H^{1}_{0}(\Omega)} \leq 1  } \left|  \left< e\B \bfu_m(t) ,\nabla w \right>      -\left<g(t),w\right>     \right|^2 \\ 
\leq &\ \tilde{M}  \sup_{\| w\|_{H^{1}_{0}(\Omega)} \leq 1  } \left( \left|  \left< e\B \bfu_m(t),\nabla w \right>  \right| + \left| \left<g(t) , w \right>           \right|\right)^2  \\ 
\leq & \ \tilde{M}  \sup_{\| w\|_{H^{1}_{0}(\Omega)} \leq 1  } \left( 2\underbrace{\left| \left< e\B \bfu_m(t),\nabla w \right>  \right|^2}_{\stackrel{\textrm{\tiny{C.S.}}}{\leq} 
\| e\B \bfu_m(t) \|_{L^2(\Omega)}^2 \cdot \|  w \|_{H^1_{0}(\Omega)}^2   }
 +2 \left| \left<g(t),w \right>           \right|^2   \right)  \\
 \leq &\ 2M \left( \|e\B\bfu_m(t)\|^2_{L^2(\Omega)}  + \|g(t)\|^2_{H^{-1}(\Omega)} \right) 
 \end{aligned}
\end{equation}

Furthermore, for $t=0$ we get

\begin{equation}
\begin{aligned}
 \|\phi_m(0)\|_{H^1_{0}(\Omega)}^2 \leq & 2M \left( \|e\B \bfu_m(0)\|_{L^2(\Omega)}^2 + \| g(0) \|_{H^{-1}(\Omega)}^2\right)  \\
\end{aligned}
\label{phi0fuerunique}
\end{equation}
Hence, we obtain 
$$
 \|\phi_m(0)\|^2_{H^1_{0}(\Omega)} \leq C_5\left( \|(\bfu_0)_m\|^2_{H^1_\B(\Omega)} + \|\phi^e(0)\|^2_{H^{1/2}(\Gamma_e)}    \right).
$$

Finally, from the Gronwall inequality we can thus deduce
\begin{equation}
\begin{aligned}
\|\bfdu_m\|^2_{L^\infty(0,T;L^2(\Omega))} +\|\bfu_m\|^2_{L^\infty(0,T;H^1_\B(\Omega))} +\|\phi_m\|^2_{L^\infty(0,T;H^1_{0}(\Omega))} \\
\leq
C_6 \left(  \|(\bfu_1)_m\|^2_{L^2(\Omega)}  +  \|(\bfu_0)_m\|^2_{H^1_\B(\Omega)} +\|\phi^e\|^2_{L^\infty(0,T;H^{1/2}(\Gamma_e))}  \right).
\end{aligned}
\label{gronwall_folgerung}
\end{equation}

Now knowing that all these values are finite we can deduce from Eq.~\eqref{eq_160} with $\beta >0 $ that also
\begin{equation}
 \beta\| \mathcal{B} \bfdu_m \|_{L^2(0,T;L^2(\Omega))} < \infty.
\label{Bbfdu_beschr}
\end{equation}

It now remains to show that $\|\bfddu_m\|_{L^{2}(0,T;(H^{1}_\calB(\Omega))')}$ is finite.
We follow the general guideline given in e.g. \cite[p. 355]{evans}.

Fix any $\tilde{v}\in \HB{1}$ with $\|\tilde{v}\|_{\HB{1}} \leq 1$ and $\tilde{v}:=\tilde{v}^1+\tilde{v}^2$ with 
$\tilde{v}^1\in \operatorname{span}\{\bfv_i\}_{i=1}^m$ and $\left< \tilde{v}^2,\bfv_i \right>=0$ for all $1 \leq i \leq m$.
Since $\{ \bfv_i\}_{i=0}^m$ can be assumed orthogonal in $\HB{1}$, $$\|\tilde{v}^1\|_{\HB{1}} \leq \|\tilde{v}\|_{\HB{1}} \leq 1.$$
Now with $\bfu_m=\sum_{i=0}^m u_m^i(t) \bfv_i$ the following holds almost everywhere in $[0,T]$:
\begin{equation*}
\begin{aligned}
\left< \bfddu_m(t),\tilde{v} \right>_{H^1_\B(\Omega)}=\left< \bfddu_m(t),\tilde{v} \right>&=\left< \bfddu_m(t),\tilde{v}^1 \right>\\
& =  \left<f(t), \tilde{v}^1\right> 
  -\left< c^E \calB \bfu_m(t) ,\calB \tilde{v}^1\right> -\left<e^T\nabla \phi_m(t),\calB \tilde{v}^1  \right> \\
& \quad -\alpha\left<\rho \bfdu_m(t),\tilde{v}^1  \right> - \beta\left< c^E \calB \bfdu_m(t) ,\calB \tilde{v}^1\right>,
\end{aligned}
\end{equation*}
where the subscript $H^1_\B(\Omega)$ denotes the
duality pairing between $\left( \HB{1}\right)'$ and $\HB{1}$.

Using the Cauchy--Schwarz inequality we can deduce
\begin{equation}
\begin{aligned}
&\left|\left< \bfddu_m(t),\tilde{v} \right>_{H^1_{\calB}(\Omega)}\right| \\
&\leq C_7\left(  \|f(t)\|_{(H^1_\B(\Omega))'} 
 + \|\bfu_m(t)\|_{H^1_{\calB}(\Omega)} +\|\bfdu_m(t)\|_{L^2(\Omega)} + \|\calB \bfdu_m(t) \|_{L^2(\Omega)} \right)\\
& \quad+\left|\left<e^T\nabla \phi_m(t),\calB \tilde{v}^1  \right>\right|. 
\label{upp}
\end{aligned}
\end{equation}

Using the Lax--Milgram lemma again on the form Eq.~\eqref{bilinear_form} for an arbitrary $t\in[0,T]$ we can further deduce with analogous arguments as in Eq.~\eqref{lax_milgram}
that the following holds
$$
\left|\left<e^T\nabla \phi_m(t),\calB \tilde{v}^1  \right>_{H^1_{\calB}(\Omega)}\right| \leq M 
\left( \|\calB \bfu_m(t)\|^2_{L^2(\Omega)}+        \|\phi^e(t)\|^2_{H^{1/2}(\Gamma_e)} \right).
$$
Thus by repetitive application of the Young inequality we get for the norm
\begin{equation*}
\begin{aligned}
\|\bfddu_m(t)\|^2_{(H^{1}_\calB(\Omega))'} =&  \sup_{ \|v\|_{H^{1}_{\calB}(\Omega)} \leq 1  }  \left|\left< \bfddu_m(t),v \right>\right|^2_{H^1_{\calB}(\Omega)} \\
\leq&C_8\Big( \|f(t)\|^2_{(H^1_\B(\Omega))'} + \|\bfu_m(t)\|^2_{H^1_{\calB}(\Omega)} +\|\bfdu_m(t)\|^2_{L^2(\Omega)}\\
&+ \|\calB \bfdu_m(t) \|^2_{L^2(\Omega)} +\|\calB \bfu_m(t)\|^2_{L^2(\Omega)}+       \|\phi^e(t)\|^2_{H^{1/2}(\Gamma_e)} \Big).
\end{aligned}
\end{equation*}

Now we have finiteness for all components, hence we can finally integrate inequality Eq.~\eqref{upp} over $[0,T]$.

We rearrange the terms and apply the estimates Eq.~\eqref{gronwall_folgerung} and Eq.~\eqref{Bbfdu_beschr}.
\begin{equation}
\begin{aligned}
&\int_{0}^T \|\bfddu_m(s)\|^2_{(H^{1}_\calB(\Omega))'}ds \\
 \leq &\ C_9 \int_{0}^T \Big(\ \|f(s)\|^2_{(H^1_\B(\Omega))'} +  \|\bfu_m(s)\|^2_{H^1_{\calB}(\Omega)} +   
\|\bfdu_m(s)\|^2_{L^2(\Omega)}+ \|\calB \bfdu_m(s) \|^2_{L^2(\Omega)}\\
&\quad \quad +\|\phi^e(s)\|^2_{H^{1/2}(\Gamma_e)} \Big)\
\, ds \\
 \leq&\ C_{10} \left( \|(\bfu_0)_m\|^2_{H^1_{\calB}(\Omega)  }  + \|(\bfu_1)_m\|^2_{L^2(\Omega)  } + \| \phi^e  \|^2_{H^1(0,T;H^{1/2}(\Gamma_e))   }  \right)
\end{aligned}
\label{upp_energy}
\end{equation}

Thus, it is now clear that
\begin{equation}
\begin{aligned}
(\bfu_m,\phi_m) &\in L^\infty(0,T;H^1_\B(\Omega)) \times L^\infty(0,T;H^1_{0}(\Omega)), \\
\bfdu_m        &\in L^\infty(0,T;L^2(\Omega)),   \\
\bfddu_m       &\in L^2(0,T;(H^{1}_\calB(\Omega))').
\end{aligned}
\label{energy_estimates1}
\end{equation}

\textbf{Phase 3: Weak limit}\\
Following e.g. \cite[p.~384]{evans}, \cite[p.~239]{schweizer} from the energy estimates Eq.~\eqref{gronwall_folgerung} and Eq.~\eqref{upp_energy} we get the boundedness of the sequences
\begin{equation*}
\begin{aligned}
(\bfu_m,\phi_m)_{m=1}^\infty &\ \text{in}\ L^\infty(0,T;H^1_\B(\Omega)) \times L^\infty(0,T;H^1_{0}(\Omega)),\\
(\bfdu_m)_{m=1}^\infty& \ \text{in}\ L^\infty(0,T;L^2(\Omega))\  \text{and}\\
(\bfddu_m )_{m=1}^\infty&\ \text{in}\ L^2(0,T;(H^{1}_\calB(\Omega))').
\end{aligned}
\end{equation*}
Thus there exist subsequences
\begin{equation*}
 (\bfu_{m_l},\phi_{m_l})_{l=1}^\infty \subseteq (\bfu_m,\phi_m)_{m=1}^\infty,\ 
(\bfdu_{m_l})_{l=1}^\infty \subseteq (\bfdu_{m})_{m=1}^\infty \ \text{and}\  
(\bfddu_{m_l} )_{l=1}^\infty \subseteq (\bfddu_{m} )_{m=1}^\infty,
\end{equation*}

with $(\bfu, \phi)\in L^\infty(0,T;H^1_\B(\Omega)) \times L^\infty(0,T;H^1_{0,\Gamma}(\Omega))$, 
$\bfdu        \in L^\infty(0,T;L^2(\Omega))$ and
$\bfddu       \in L^2(0,T;(H^{1}_\calB(\Omega))')$
such that 
\begin{equation}
\begin{aligned}
\bfu_{m_l}&\rightharpoonup     \bfu      \text{ weakly-* in } L^\infty(0,T;H^1_\B(\Omega)),  \\
\phi_{m_l}&\rightharpoonup     \phi_0      \text{ weakly-* in }  L^\infty(0,T;H^1_{0}(\Omega)), \\
\bfdu_{m_l}&\rightharpoonup    \bfdu     \text{ weakly-* in } L^\infty(0,T;L^2(\Omega)),   \\
\bfddu_{m_l}&\rightharpoonup   \bfddu    \text{ weakly in } L^2(0,T;(H^{1}_\calB(\Omega))').\\
\end{aligned}
\label{energy_estimates2}
\end{equation}
We now proceed to show that the weak limit is a solution of the weak form.
Following \cite[p. 384]{evans} we fix a $N\in \N$ and choose functions $\bfv\in C^1(0,T; H^1_\B(\Omega))$ and $w\in C^1(0,T;H^1_0(\Omega))$ 
having the form

\begin{equation}
\bfv(t):=\sum_{k=1}^N u^k_m(t)\bfv_k, \quad w(t):=\sum_{k=1}^N \phi^k_m(t)w_k.
\label{weak_discrete}
\end{equation}
We choose $m\geq N$, multiply the discretized versions for each pair $(\bfv_k, w_k)$  of the weak form Eq.~\eqref{eq_schwache_form} with $(u^k_m(t), \phi^k_m(t))$, 
sum over $k=1,\dots,N$, integrate with respect to $t$. This yields
\begin{equation}
\begin{aligned}
&\int_0^T \Big(
\int_\Omega \rho \bfddu_m(s)^T \bfv \,d\Omega +\alpha \int_{\Omega} \rho \bfdu_m(s)^T\bfv \,d\Omega +\int_\Omega \left(c^E \calB \bfu_m(s)\right)^T\calB \bfv \,d\Omega   \\
\,&\quad \quad + \beta\int_\Omega \left(c^E \calB \bfdu_m(s)\right)^T\calB \bfv \,d\Omega + \int_\Omega \left( e^T\nabla\phi_m(s)\right)^T \calB \bfv \,d\Omega  \\
\,&\quad \quad+ \int_\Omega \left( e \calB \bfu_m(s)\right)^T\nabla w \,d\Omega -  \int_\Omega \left(\epsilon^S \nabla \phi_m(s)\right)^T\nabla w \,d\Omega\Big)\, ds \\
  \,=& \int_0^T \left<f(s),\bfv \right>\, ds +\int_0^T \left<g(s),w \right>\, ds.
		\label{eq_lahmer_178}
\end{aligned}
\end{equation}
Fixing $m=m_l$ and using Eq.~\eqref{energy_estimates2} we obtain in the limit $m \rightarrow \infty$ along the subsequence $m_l$
\begin{equation}
\begin{aligned}
&\int_0^T \Big(
\int_\Omega \rho \bfddu(s)^T \bfv \,d\Omega +\alpha \int_{\Omega} \rho \bfdu(s)^T\bfv \,d\Omega +\int_\Omega \left(c^E \calB \bfu(s)\right)^T\calB \bfv \,d\Omega \\
\,& \quad \quad + \beta\int_\Omega \left(c^E \calB \bfdu(s)\right)^T\calB \bfv \,d\Omega + \int_\Omega \left( e^T\nabla\phi_0(s)\right)^T \calB \bfv \,d\Omega \\
\,&\quad \quad + \int_\Omega \left( e \calB \bfu(s)\right)^T\nabla w \,d\Omega -  \int_\Omega \left(\epsilon^S \nabla \phi_0(s)\right)^T\nabla w \,d\Omega\Big)\, ds \\
  \,=& \int_0^T \left<f(s),\bfv \right>\, ds +\int_0^T \left<g(s),w \right>\, ds.
	\label{eq_lahmer_179}
\end{aligned}
\end{equation}
Noting that all functions of form Eq.~\eqref{weak_discrete} are dense in the according spaces this equality holds for all functions $\bfv\in L^2(0,T;H^1_\B(\Omega)),\, w\in L^2(0,T;H^1_{0}(\Omega))$.
In particular it follows that also
\begin{equation}
\begin{aligned}
&\int_\Omega \rho \bfddu^T \bfv \,d\Omega +\alpha \int_{\Omega} \rho \bfdu^T\bfv \,d\Omega +\int_\Omega \left(c^E \calB \bfu\right)^T\calB \bfv \,d\Omega   + \beta\int_\Omega \left(c^E \calB \dot{\bfu}\right)^T\calB \bfv \,d\Omega  \\
\,&+ \int_\Omega \left( e^T\nabla\phi_0\right)^T \calB \bfv \,d\Omega + \int_\Omega \left( e \calB \bfu\right)^T\nabla w \,d\Omega -  \int_\Omega \left(\epsilon^S \nabla \phi_0\right)^T\nabla w \,d\Omega \\
  \,=& \left<f,\bfv \right> + \left<g,w \right>
	\label{mein_249}
\end{aligned}
\end{equation}
almost everywhere $t\in[0,T]$ for all $\bfv \in H^1_\calB(\Omega)$ and $w\in H^1_{0}(\Omega)$. 

Following \cite{lahmer} and \cite[section 7.2.2 Thm. 3, p. 385]{evans} we confirm that the initial conditions are also met.
Choose any function $(\bfv,0)$ with $\bfv \in C^2([0,T];H^1_\calB(\Omega))$ and $\bfv(T)=\dot{\bfv}(T)=0$. 
By integrating by parts twice with respect to $t$ of Eq.~\eqref{eq_lahmer_178} we get

\begin{equation}
\begin{aligned}
&\int_0^T \Big(
\int_\Omega \rho \bfu_m(t)^T \bfddv \,d\Omega -\alpha \int_{\Omega} \rho \bfu_m(t)^T\bfdv \,d\Omega +\int_\Omega \left(c^E \calB \bfu_m(t)\right)^T\calB \bfv \,d\Omega  \\
\,&\quad \quad - \beta\int_\Omega \left(c^E \calB \bfu_m(t)\right)^T\calB \bfdv \,d\Omega  
+ \int_\Omega \left( e^T\nabla\phi_m(t)\right)^T \calB \bfv \,d\Omega \Big)\, dt \\
  \,=& 
	\int_0^T \left<f(t),\bfv \right>\, dt - \left<\rho \bfu_m(0),\bfdv(0)\right> +\left<\rho \bfdu_m(0),\bfv(0)\right> \\
\,&\quad \quad +\alpha \left<\rho\bfu_m(0),\bfv(0) \right> +\beta \left<c^E\B\bfu_m(0),\B\bfv(0) \right>
\label{eq_lahmer_182}
\end{aligned}
\end{equation}
and analogously using Eq.~\eqref{eq_lahmer_179} we get
\begin{equation}
\begin{aligned}
&\int_0^T \Big(
\int_\Omega \rho \bfu(t)^T \bfddv \,d\Omega -\alpha \int_{\Omega} \rho \bfu(t)^T\bfdv \,d\Omega +\int_\Omega \left(c^E \calB \bfu(t)\right)^T\calB \bfv \,d\Omega \\  
\,&\quad \quad - \beta\int_\Omega \left(c^E \calB \bfu(t)\right)^T\calB \bfdv \,d\Omega +\int_\Omega \left( e^T\nabla\phi_0(t)\right)^T \calB \bfv \,d\Omega \Big)\, dt \\
  \,=& 
	\int_0^T \left<f(t),\bfv \right>\, dt - \left<\rho \bfu(0),\bfdv(0)\right> +\left<\rho \bfdu(0),\bfv(0)\right> \\
\,&\quad \quad +\alpha \left<\rho\bfu(0),\bfv(0) \right> +\beta \left<c^E\B\bfu(0),\B\bfv(0) \right>.
		\label{eq_lahmer_181}
\end{aligned}
\end{equation}
For Eq.~\eqref{eq_lahmer_182} we set $m=m_l$ and recall Eq.~\eqref{energy_estimates2} to deduce
\begin{equation}
\begin{aligned}
&\int_0^T \Big(
\int_\Omega \rho \bfu(t)^T \bfddv \,d\Omega -\alpha \int_{\Omega} \rho \bfu(t)^T\bfdv \,d\Omega +\int_\Omega \left(c^E \calB \bfu(t)\right)^T\calB \bfv \,d\Omega \\  
\,&\quad \quad- \beta\int_\Omega \left(c^E \calB \bfu(t)\right)^T\calB \bfdv \,d\Omega + \int_\Omega \left( e^T\nabla\phi_0(t)\right)^T \calB \bfv \,d\Omega \Big)\, dt \\
  \,=& 
	\int_0^T \left<f(t),\bfv \right>\, dt - \left<\rho \bfu_0,\bfdv(0)\right> +\left<\rho \bfu_1,\bfv(0)\right> \\
\,&\quad \quad +\alpha \left<\rho\bfu_0,\bfv(0) \right> +\beta \left<c^E\B\bfu_0,\B\bfv(0) \right>.
		\label{eq_lahmer_183}
\end{aligned}
\end{equation}
By equating coefficients of Eq.~\eqref{eq_lahmer_181} and Eq.~\eqref{eq_lahmer_183} (set either $\bfv(0)$ or $\bfdv(0)$ to zero) we conclude $\bfu(0)=\bfu_0$ and $\bfdu(0)=\bfu_1$.

\textbf{Phase 4: Uniqueness}\\

Following e.g. \cite[p. 385]{evans} it suffices to show that the only weak solution 
with $$f \equiv 0,\, g\equiv 0,\,  \phi^e \equiv 0,\, \bfu_0=\bfu_1\equiv 0$$ is 
$$
\bfu\equiv 0,\, \phi \equiv 0.
$$
Notice that by property Eq.~\eqref{Bbfdu_beschr}  $\| \mathcal{B} \bfdu_m \|_{L^2(0,T;L^2(\Omega))} $ is finite.
Hence, the remark in \cite[remark below Thm 4, section 7.2.2 c), p. 385]{evans} does not apply to our case and we can continue in the fashion of \cite[Thm. 4, section 7.1.2c),   p. 358]{evans} instead.
Passing to limits, we substitute $\bfv=\bfu$ and $w=\phi_0$ in the original weak form. 
This is not prohibited as by property Eq.~\eqref{Bbfdu_beschr} all components exist also in the limit.
Hence, we can deduce that the following non-discretized inequality holds
\begin{equation*}
\begin{aligned}
&\ C_1\left(\|\bfdu(t)\|^2_{L^2(\Omega)} +\|\bfu(t)\|^2_{H^1_\B(\Omega)} +\|\phi_0(t)\|^2_{H^1_{0}(\Omega)} \right)\\
\leq& \  C_3\left(  \|\bfdu(0)\|^2_{L^2(\Omega)} + \|\bfu(0)\|^2_{H^1_\B(\Omega)}  +\|\phi_0(0)\|^2_{H^1_{0}(\Omega)}   \right)\\
&+C_4\int_0^t \left(  \|\bfdu(s)\|^2_{L^2(\Omega)} + \|\bfu(s)\|^2_{H^1_\B(\Omega)}  +\|\phi_0(s)\|^2_{H^1_{0}(\Omega)}   \right)ds\\
&+ \|f\|^2_{L^2(0,T;(H^{1}_\calB(\Omega))')}  + \|g\|^2_{H^1(0,T;H^{-1}(\Omega))}.
\end{aligned}
\end{equation*}
In the case $t=0$ we get from Eq.~\eqref{phi0fuerunique} that $\|\phi(0)\|^2_{H^1_0(\Omega)}=0$.
Hence, we now note that 
$$\tilde{C_2}= \frac{1}{C_1}\left(C_3 \eta(0) + \|f\|^2_{L^2(0,T;(H^{1}_\calB(\Omega))')} +\|g\|^2_{H^1(0,T;H^{-1}(\Omega))}\right) =0. $$
Finally, we can apply the second part of the Gronwall inequality to conclude that 
$$
\eta(t)=\|\bfdu(t)\|^2_{L^2(\Omega)} +\|\bfu(t)\|^2_{H^1_\B(\Omega)} +\|\phi_0(t)\|^2_{H^1_{0}(\Omega)} =0 \quad \text{ a.e. } t\in[0,T].
$$
 
Thus, the only solution can be the trivial solution.
\end{proof}
Through the theorem we know what requirements we need to get existence of a solution of the weak form. Now prerequisites can be derived to achieve higher regularities of the solutions.\\
The following theorem is inspired by Thm. 5, chapter 7.2 in \cite{evans}. 
The proof uses ideas from \cite{leugering} adapted for additional Rayleigh damping.
\begin{thm}
Let all requirements of Thm.~\ref{theorem1} hold.
If additionally 
$\bfu_0 \in H^2(\Omega)$,
$\bfu_1 \in H^1(\Omega)$,
$\beta\bfu_1 \in H^2(\Omega)$,
$\phi^e \in H^2(0,T;H^{1/2}(\Gamma_e))$, 
then 
\begin{equation}
\begin{aligned}
\bfu         \in L^\infty(0,T;H^1_\B(\Omega)),\quad 
\bfdu        &\in L^\infty(0,T;H^1_\B(\Omega)),\quad
\bfddu       \in L^\infty(0,T;L^2(\Omega)),\\
\phi         \in L^\infty(0,T;H^1_{0,\Gamma}(\Omega)),\quad
\dot{\phi}   &\in L^\infty(0,T;H^1_{0,\Gamma}(\Omega)). \\
\end{aligned}
\end{equation}
\label{theorem2}
\end{thm}
\begin{proof}

We differentiate the weak form Eq.~\eqref{eq_schwache_form} once with respect to time $t$ and test the result first with $(\bfddu_m(t),0)$ to obtain
\begin{equation*}
\begin{aligned}
\frac{1}{2}\frac{d}{dt} \bigg( \left<\rho \bfddu_m(t), \bfddu_m(t) \right>  + \left< c^{E}\B\bfdu_m(t),\B\bfdu_m(t)\right>\bigg) 
+ \alpha \left<\rho \bfddu_m(t), \bfddu_m(t) \right>  \\
+\beta \left<c^E \B\bfddu_m(t), \B\bfddu_m(t) \right>+ \left<e \nabla \dphi_m(t),\B \bfddu_m(t) \right>  
= \left<\dot{f}(t),\bfddu_m(t) \right>.
\end{aligned}
\end{equation*}
Then we differentiate the weak form Eq.~\eqref{eq_schwache_form} twice with respect to time $t$ and test the result first with $(0,\dphi_m(t))$ to obtain
\begin{equation*}
\begin{aligned}
\left< e\B\bfddu_m(t),\nabla \dphi_m(t) \right>- \frac{1}{2}\frac{d}{dt}\left< \epsilon^S \nabla \dphi_m(t),\nabla \dphi_m(t)\right> = \left< \ddot{g}(t),\dot{\phi}_m(t)\right>.
\end{aligned}
\end{equation*}
Analogously to the proof of Thm.~\ref{theorem1} we subtract these two results and integrate with respect to $t$ to obtain in analogy to Eq.~\eqref{eq_160}
\begin{equation}
\begin{aligned}
 & \left(  \left<\rho \bfddu_m(t), \bfddu_m(t)\right> +\left<c^E \B \bfdu_m(t), \B \bfdu_m(t)\right>  
+\left<\epsilon^S \nabla \dphi_m(t) ,\nabla \dphi_m(t)\right>   \right) \\
&+ 2 \alpha \int_0^t\left<\rho \bfddu_m(s), \bfddu_m(s)\right> \, ds+2 \beta \int_0^t\left<c^E \B \bfddu_m(s), \B \bfddu_m(s)\right>\,ds\\
=&
 \left(  \left<\rho \bfddu_m(0), \bfddu_m(0)\right> +\left<c^E \B \bfdu_m(0), \B \bfdu_m(0)\right>  +\left<\epsilon^S \nabla \dphi_m(0) ,\nabla \dphi_m(0)\right>   \right)\\
&+ 2\int_0^t\left< \dot{f}(s), \bfdu_m(s)\right>ds -2\int_0^t\left< \ddot g(s),\dphi_m(s)\right>ds\\
\end{aligned}
\label{eq_160_neu}
\end{equation}
or, again, abbreviated as $\mathcal{F}_l=\mathcal{F}_r$.
Analogously to inequality Eq.~\eqref{eq_214} we then can obtain
\begin{equation}
\begin{aligned}
&\ C_1\left(\|\bfddu_m(t)\|^2_{L^2(\Omega)} +\|\bfdu_m(t)\|^2_{H^1_\B(\Omega)} +\|\dphi_m(t)\|^2_{H^1_{0}(\Omega)} \right) -c_{mech} \|\bfdu_m(t)\|^2_{L^2(\Omega)}  \\
\leq& \ C_3\left(  \|\bfddu_m(0)\|^2_{L^2(\Omega)} + \|\bfdu_m(0)\|^2_{H^1_\B(\Omega)}  +\|\dphi_m(0)\|^2_{H^1_{0}(\Omega)}   \right)\\
&+\ C_4\int_0^t \left(  \|\bfddu_m(s)\|^2_{L^2(\Omega)}+  \|\bfdu_m(s)\|^2_{H^1_\B(\Omega)}    +\|\dphi_m(s)\|^2_{H^1_{0}(\Omega)}   \right)ds\\
&+\ \|f\|^2_{H^1(0,T;(H^{1}_\calB(\Omega))')}  + \|g\|^2_{H^2(0,T;H^{-1}(\Omega))}    \\
\end{aligned}
\label{eq_214_neu}
\end{equation}
for some $C_1,C_3,C_4>0$.	
Note that by deriving the weak form which we then test by $(0,\phi_m(t))$ we additionally obtain a bilinear form similar to Eq.~\eqref{bilinear_form} and can analogously deduce with the Lax--Milgram lemma that
\begin{equation*}
\begin{aligned}
 \|\dphi_m(0)\|_{H^1_{0}(\Omega)}^2 \leq & 2M \left( \|e\B \bfdu_m(0)\|_{L^2(\Omega)}^2 + \| \dot g(0) \|_{H^{-1}(\Omega)}^2\right). 
\end{aligned}
\end{equation*}
This is only possible because of the added requirement of increased regularity of $\bfdu(0)$ and $\dot g$.
Furthermore, by the additional requirements on $\bfu_m(0) \in H^2(\Omega)$ we also obtain (estimating the $H^2$ norm by the norm of the Laplacian, see e.g. \cite{stein})
\begin{equation*}
\begin{aligned}
 \|\phi_m(0)\|_{H^2(\Omega)} \leq C\left(\|\bfu_m(0)\|_{H^2(\Omega)} +   \| {g(0)}\|_{H^{-1}(\Omega)} \right). \\
\end{aligned}
\end{equation*}

In order to utilize the Gronwall lemma we are left to show finiteness of \\ $\|\bfddu_m(0)\|^2_{L^2(\Omega)}$.
Notice that by the increased regularity of $\bfu_m(0)\in H^2(\Omega)$ and $\phi_m(0)\in H^2(\Omega)$ the weak solution is also a \emph{strong} solution, i.e., not quite a \emph{classical} solution but solves the classical equations in $t=0$ almost everywhere, see e.g.~\cite[section 2.3 and 3.5 ]{larsson}.
Thus, by evaluating the strong system in $t=0$ and using the initial data and previously deduced inequalities we obtain
\begin{equation*}
\begin{aligned}
\|\rho \bfddu_m(0)\|_{L^2(\Omega)} &= \|\alpha \rho \bfdu_m(0)-\calB^T \left( c^E \calB \bfu_m(0) + \beta c^E \calB \bfdu_m(0) + e^T \nabla\phi_m(0)\right)\\
\,& \quad  +{f(0)} \|_{L^2(\Omega)} \\
                                 &\leq \|\alpha \rho \bfdu_m(0)\|_{L^2(\Omega)}+\|\calB^T  c^E \calB \bfu_m(0) \|_{L^2(\Omega)} \\
                                 & \, \quad     
+ \|\calB^T \beta c^E \calB \bfdu_m(0)\|_{L^2(\Omega)} + \|\calB^Te^T \nabla\phi_m(0)  \|_{L^2(\Omega)} + \|{f(0)}\|_{L^2(\Omega)}.\\ 
\end{aligned}
\end{equation*}
Note that this $f$ is given by the Dirichlet ĺift ansatz for the strong system. Therefor we choose $\chi \in H^2(\Omega)$ where $\chi\vert_{\Gamma_g}=0$ and $\chi\vert_{\Gamma_e}=1$. With this requirement the right hand sight of the above inequality is bounded independently of $m$.

Since all components are finite, analogously to inequality Eq.~\eqref{eq_166} with 
$$
\eta(t):=\|\bfddu_m(t)\|^2_{L^2(\Omega)}+\|\bfdu_m(t)\|^2_{H^1_\B(\Omega)}  +\|\dphi_m(t)\|^2_{H^1_{0}(\Omega)}
$$
we can apply the Gronwall lemma to obtain that

\begin{equation*}
\begin{aligned}
&\|\bfddu_m(t)\|^2_{L^2(\Omega)} +\|\bfdu_m(t)\|^2_{H^1_\B(\Omega)} +\|\dphi_m(t)\|^2_{H^1_{0}(\Omega)} \\
\leq & \left(  \frac{\tilde{C_3}}{C_1} +\frac{C_4 \tilde{C_3}}{C_1^2}te^{ \frac{C_4}{C_1}t }  \right)
     \Big( \|\bfddu_m(0)\|^2_{L^2(\Omega)} +\|\bfdu_m(0)\|^2_{H^1_\B(\Omega)} +\|\dphi_m(0)\|^2_{H^1_{0}(\Omega)} \\
&+  \|f\|^2_{H^1(0,T;(H^{1}_\calB(\Omega))')}  + \|g\|^2_{H^2(0,T;H^{-1}(\Omega))}     \Big)
\end{aligned}
\end{equation*}
holds almost everywhere in $[0,T]$. 

Using results from Thm.~\ref{theorem1} it is now clear that 
\begin{equation}
\begin{aligned}
\bfu         \in L^\infty(0,T;H^1_\B(\Omega)),\quad
\bfdu        &\in L^\infty(0,T;H^1_\B(\Omega)),\quad
\bfddu       \in L^\infty(0,T;L^2(\Omega)),\\
\phi         \in L^\infty(0,T;H^1_{0,\Gamma}(\Omega)),\quad 
\dot{\phi}   &\in L^\infty(0,T;H^1_{0,\Gamma}(\Omega)). 
\end{aligned}
\end{equation}
\end{proof}
\begin{rem}
One may think that in order to achieve $\bfddu \in$ $L^\infty(0,T;L^2(\Omega))$ it is only required that $\bfdu(0)=\bfu_1 \in H^1(\Omega)$ instead of $\beta\bfu_1 \in H^2(\Omega)$ (or more precisely $\|\B^T \beta c^E\B\bfu_1\|_{L^2}(\Omega))< \infty$).
However, this is not the case.

The condition $\beta \bfu_1 \in H^2(\Omega)$ is required to show that $\|\bfddu_m(0)\|_{L^2(\Omega)}$ is finite, such that the Gronwall inequality can be applied.
\end{rem}
\begin{rem}
In e.g. \cite[p.~390, Eq.~(59)]{evans} a $H^2$ regularity for $\bfu$ is achieved by selecting the test functions for $\bfu$ to be the complete eigenfunction sequence of $-\Delta \bfu$ which, indirectly, allows an estimation of $\|\bfu\|_{H^2(\Omega)}$. 
A similar argument should also be possible for $\B^T \B$ (or more precisely the operator that works on the solution vector $(\mathbf{u},\phi)^T$ 
and contains $\B^T \B$). 
This would directly increase the regularity of $\phi$ so that not only $\bfu \in L^\infty(0,T;H^2(\Omega))$ but also $\phi \in L^\infty(0,T;H^2(\Omega) \cap H^1_{0,\Gamma}(\Omega)) $.

However, the authors did not follow that argumentation. 
Note that the here occurring differential operators are slightly different from the Laplacian. 
Hence, this leads to rather unpleasant changes due to the now very technical arguments and spaces.
In that case, it would be possible to reduce the regularity requirements, however this would also change the resulting spaces and increase the cost of technical proof steps.
\end{rem}

With the estimations and equations in Theorem \ref{theorem1} and the corresponding proof, a long-time behavior of the energy function $\eta$ and in particular of each component can be derived.
\begin{cor}
Let all requirements of Thm.~\ref{theorem1} hold and let $\alpha,\beta >0$ strictly.
If additionally there exists a $t_0\in \R, \, t_0 \ge 0$ such that  $\phi^e(t) = 0$ for $t\ge t_0$,
then 
$$\Vert \dot{\mathbf{u}}_m(t)\Vert_{L^2(\Omega)}\rightarrow 0,\quad \Vert \mathcal{B}\mathbf{u}_m(t)\Vert_{L^2(\Omega)}\rightarrow 0,\quad \Vert \phi_m(t)\Vert_{H^1_0(\Omega)}\rightarrow 0$$
for $t\rightarrow \infty$.\\
Furthermore the energy of the system 
 $$\eta(t)=\Vert \dot{\mathbf{u}}_m(t)\Vert^{2}_{L^2(\Omega)}+\Vert \mathbf{u}_m(t)\Vert^{2}_{H^1_{\mathcal{B}}(\Omega)}+\Vert \phi_m(t)\Vert^2_{H^1_{0}(\Omega)}$$
converges to a constant $\eta(t) \rightarrow c\in \mathbb{R}^{+}$ for $t\rightarrow \infty$.
\end{cor}
\begin{proof}
The right hand side $\mathcal{F}_r(t)$ of the energy balance Eq.~\eqref{eq_160} is constant for $t \geq  t_0$ as no new energy is given into the system starting from time $t_0$, i.e. $\mathcal{F}_r(t)=c_1 \in \R^{ \ge 0}$ for $t\ge t_0$.
Let 
$$\gamma(t):=2 \alpha \int_0^t\left<\rho \bfdu_m(s), \bfdu_m(s)\right> ds+2 \beta \int_0^t\left<c^E \B \bfdu_m(s), \B \bfdu_m(s)\right> ds,$$

and let 
\begin{align}
 \tilde{\eta}(t):= \left<\rho \bfdu_m(t), \bfdu_m(t)\right> +\left<c^E \B \bfu_m(t), \B \bfu_m(t)\right> + \left<\epsilon^S \nabla \phi_m(t) ,\nabla \phi_m(t)\right>.
 \end{align}\label{etatilde}

Then  Eq.~\eqref{eq_160} implies that $ \tilde{\eta}(t) +\gamma(t)=c_1$ for $t \ge t_0$.
As $\gamma(t)$ is monotonically increasing, it follows that $\tilde{\eta}(t)$ is monotonically decreasing.
Both $\tilde{\eta}(t), \gamma(t)$ are bounded below and above by zero and $c_1$, respectively. 
Hence, $\gamma(t)$ and $\tilde{\eta}(t)$ must converge. 
Based on these results we get $0\leq\tilde{\eta}(t)\rightarrow c_2 \leq c_1 < \infty $ for  $t\rightarrow\infty$. 
We know that $\gamma(t)$ converges, thus the occurring integrands must converge towards 0, i.e.,
$$ \left<\rho \bfdu_m(s), \bfdu_m(s)\right>\rightarrow 0,\quad \left<c^E \B \bfdu_m(s), \B \bfdu_m(s)\right> \rightarrow 0.$$
Through these results and by utilizing positive definiteness of $c^E$ and estimations similar to Eq.~\eqref{eq_161}
we also get the following convergence result
\begin{align}\label{upunkt0}
\bfdu_m(s)\rightarrow 0 \quad \B \bfdu_m(s) \rightarrow 0.
\end{align}

For the next steps it is already known that
$$ \tilde{\eta}(t):= \underbrace{\left<\rho \bfdu_m(t), \bfdu_m(t)\right>}_{\rightarrow 0} +\left<c^E \B \bfu_m(t), \B \bfu_m(t)\right> + \left<\epsilon^S \nabla \phi_m(t) ,\nabla \phi_m(t)\right>\rightarrow c_2 \ge 0.$$ 
We have to show that one of the other two summands converges and determine the limit values.

We get the convergence of $\left<c^E \B \bfu_m(t), \B \bfu_m(t)\right>$ by taking advantage of Eq.~\eqref{upunkt0} and using the characteristic of the time derivative of $\B \bfu_m(t)$. 
As all other summands converge this implies that also $\left<\epsilon^S \nabla \phi_m(t) ,\nabla \phi_m(t)\right>$ must converge.
In order to specify the limit values we test the weak form Eq.~\eqref{eq_schwache_form} first with $(\bfu_m(t),0)$
and get 
\begin{equation}
\begin{aligned}
&\underbrace{\left< \rho \bfddu_m(t), \bfu_m(t) \right>}_{\rightarrow 0} +\underbrace{\alpha \left< \rho \bfdu_m(t),\bfu_m(t) \right>}_{\rightarrow 0} +\left< c^E \calB \bfu_m(t),\calB \bfu_m(t) \right>   \\
+& \underbrace{\beta\left< c^E \calB \bfdu_m(t),\calB \bfu_m(t) \right>}_{\rightarrow 0}  + \left<  e^T\nabla\phi_m(t), \calB \bfu_m(t) \right>  
=  \underbrace{\left<f(t),\bfu_m(t) \right>}_{= 0}
\end{aligned}
\label{eq_konvergenz_schwache_form_u}
\end{equation}
to obtain
\begin{equation}
\begin{aligned}
\lim_{t \to \infty} \left<  e^T\nabla\phi_m(t), \calB \bfu_m(t) \right> = \lim_{t \to \infty}
-\left< c^E \calB \bfu_m(t),\calB \bfu_m(t) \right> \leqslant 0
\end{aligned}
\label{eq_konvergenz1}
\end{equation}
and then a second time with $(0,\phi_m(t))$
\begin{equation*}
\begin{aligned}
\left<  e \calB \bfu_m(t),\nabla \phi_m(t) \right> -  \left< \epsilon^S \nabla \phi_m(t),\nabla \phi_m(t) \right> =   \underbrace{\left<g(t),\phi_m(t) \right> }_{=0}
\end{aligned}
\end{equation*}
yielding
\begin{equation}
\begin{aligned}
\lim_{t \to \infty}\left<  e \calB \bfu_m(t),\nabla \phi_m(t) \right> = \lim_{t \to \infty} \left< \epsilon^S \nabla \phi_m(t),\nabla \phi_m(t) \right> \geqslant 0.
\end{aligned}
\label{eq_konvergenz2}
\end{equation}
From Eq.~\eqref{eq_konvergenz1} and Eq.~\eqref{eq_konvergenz2} we get $$\lim_{t \to \infty}
\left< c^E \calB \bfu_m(t),\calB \bfu_m(t) \right> = \lim_{t \to \infty} \left< \epsilon^S \nabla \phi_m(t),\nabla \phi_m(t) \right>= 0.$$

Note that from Eq.~\eqref{lax_milgram} we are aware that $\|e\B \bfu_m(t)\|_{L^2(\Omega)}\rightarrow 0$ and using the requirement for $\phi^e(t)$ implies that also $\| \phi_m(t)\|_{H^1_0(\Omega)}\rightarrow 0.$\\
It is clear that $\tilde{\eta}(t) \rightarrow 0$ and with the characteristics of the material parameters
$$ \rho >0,\quad c^E \text{ and } \epsilon^S \text{ symmetric, positive definite,}$$
we conclude
 $$\Vert \dot{\mathbf{u}}_m(t)\Vert_{L^2(\Omega)}\rightarrow 0,\quad \Vert \mathcal{B}\mathbf{u}_m(t)\Vert_{L^2(\Omega)}\rightarrow 0,\quad \Vert \phi_m(t)\Vert_{H^1_0(\Omega)}\rightarrow 0$$
for $t\rightarrow \infty$.
Finally, we know that the derivatives in time and space of $\mathbf{u}_m(t)$ converge to zero, so $\Vert \mathbf{u}_m(t) \Vert_{L^2(\Omega)}\rightarrow \tilde{c}\in \mathbb{R}$.

Then we can conclude that $\eta(t) \rightarrow c\in \mathbb{R}^{+}$ for $t\rightarrow \infty$.
\end{proof}
As expected, we also find this theorized behavior in our numerical simulation results, see also Fig. \ref{fig_langzeit}. 
There, the monotonically decreasing energy term $\tilde{\eta}$ is shown.
The electrode on the top of the piezoceramic disk is excited by the potential pulse as shown in Fig.~\ref{fig_potentialpuls}. 
The time integration is given by a HHT-method, which is commonly used for piezoceramics (see \cite{kaltenbacher_buch}). 
These results were obtained by applying our simulation tool which will be focused on in upcoming publications.
Note that small inaccuracies can occur due to numerical reasons.

\begin{rem}
By using similar techniques as in the second part of the proof of Thm.~\ref{theorem1}, the last Corollary can be extended to non-discretized solutions of the partial differential equations.
\end{rem}

 \begin{figure}[!ht]
 \begin{center}
\includegraphics[width=0.5\textwidth]{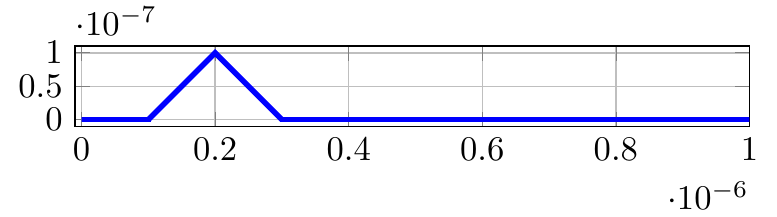}
\caption{Potential pulse for 2D transient simulation in FEniCS. }
\label{fig_potentialpuls}
\end{center}
\end{figure}

 \begin{figure}[!ht]
 \begin{center}
\includegraphics[width=0.9\textwidth]{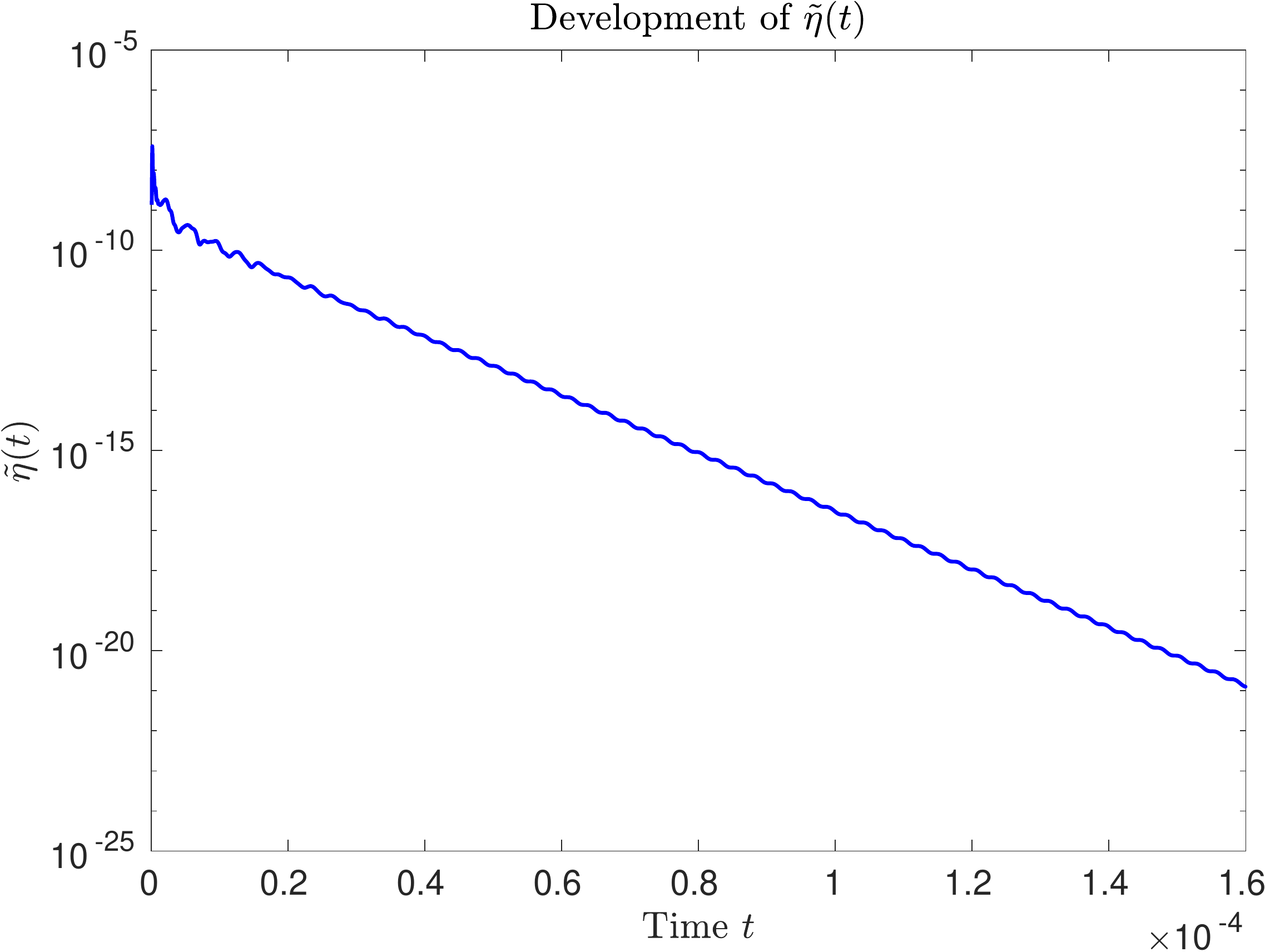}
\caption{Energy term $\tilde{\eta}(t)=\Vert \dot{\mathbf{u}}_m(t)\Vert_{L^2(\Omega)}+\Vert \mathcal{B}\mathbf{u}_m(t)\Vert_{L^2(\Omega)}+\Vert \phi_m(t)\Vert_{H^1_{0}(\Omega)}.$}
\label{fig_langzeit}
\end{center}
\end{figure}

\section{Conclusion}
Piezoelectric materials are widely diversified in their applications. 
Since measurements on real specimens are very expensive, computer simulations are used instead. 

However, in order to confidently use these computer simulations, the underlying damped partial differential equation must be analyzed.
In this paper, we prove existence, uniqueness and regularity of weak solutions of the governing partial differential equations and show some results on the long term behavior of solutions.

The obtained theoretical results are consistent with numerical results gained from a computational simulation of the model. With this, the basis is formed for ongoing design optimization of piezoelectric transducers.

\appendix
\section{Definitions}\label{appendix_definition}

Let $d,k\in \N$ be integers and let $\alpha$ be a multi-index.
Then we define the functional spaces
\begin{equation*}
\begin{aligned}
C^k(\Omega) &:=\Big\{\ \sigma: \Omega \rightarrow \R^d : \sigma \text{ is $k$-times continuously differentiable}  \Big\}\, \\
L^2(\Omega)&:=  \bigg\{\ \sigma: \Omega \rightarrow \R^d : \quad \|\sigma\|^2_{L^2(\Omega)}:=\int_{\Omega} \sigma^T\sigma \, d\Omega < \infty  \bigg\}\, \\
H^1(\Omega)&:=\Big\{\ \sigma: \Omega \rightarrow \R : \quad \|\sigma\|_{H^1(\Omega)}^2:=\|\sigma\|^2_{L^2(\Omega)} +  \|\nabla \sigma\|^2_{L^2(\Omega)} < \infty  \Big\}\, \\
H^1_0(\Omega)&:=\Big\{\ \sigma \in H^1(\Omega) : \sigma\big|_{\Gamma}=0 \text{ with }  \|\sigma\|_{H^1_0(\Omega)}:=\|\sigma\|_{H^1(\Omega)}    \Big \}\, \\
%
H^{-1}(\Omega)&:=\bigg\{\ f \text{ continuous linear functional on $H^1_0(\Omega)$ : }\\
& \quad \quad \sup_{\|\sigma\|_{H^1_0(\Omega)} \leq 1} \left|\left<f,\sigma\right>\right|  <\infty \bigg\}\ .
\end{aligned}
\end{equation*}

\begin{equation*}
\begin{aligned}
\text{Let } \sigma : [0,T] \rightarrow X& \text{ be Bochner-measurable. Then} \\
L^2(0,T;X)&:=\bigg\{\ \sigma : [0,T] \rightarrow X :\quad  \int_{[0,T]} \|\sigma(t)\|_X^2 \, dt < \infty \bigg\}\, \\
L^\infty(0,T;X)&:=\Big\{\ \sigma : [0,T] \rightarrow X :\quad  \esssup_{0\leq t\leq T } \|\sigma(t)\|_X < \infty \Big\}\, \\
H^1(0,T;X)&:=\bigg\{\ \sigma : [0,T] \rightarrow X :\quad  \int_{[0,T]} \|\sigma(t)\|_X^2 + \|\dot{\sigma}(t)\|_X^2 \, dt < \infty \bigg\}\, \\
H^{2}(\Omega)&:=\Bigg\{\ \sigma: \Omega \rightarrow \R^3 : \\
& \quad \quad \|\sigma\|_{H^2(\Omega)}:= \left( \sum_{|\alpha|\leq 2} \|D^{(\alpha)}\sigma\|_{L^2(\Omega)}\right)^{1/2}  <\infty \Bigg\}\ .\\
\end{aligned}
\end{equation*}

\interlinepenalty10000 
\bibliographystyle{plain}
\bibliography{literature}

\end{document}